\documentclass[11pt,reqno]{amsart}
\usepackage{a4wide}
\usepackage{amssymb}
\usepackage{amsthm}
\usepackage{amsmath}
\usepackage{amscd}
\usepackage{verbatim,url}
\usepackage{color,cancel}
\usepackage[mathscr]{eucal}
\numberwithin{equation}{section}
\allowdisplaybreaks

\newtheorem{theorem}{Theorem}
\newtheorem{lemma}[theorem]{Lemma}

\newtheorem{conjecture}[theorem]{Conjecture}
\newtheorem{proposition}[theorem]{Proposition}

\theoremstyle{remark}

\numberwithin{theorem}{section} \numberwithin{equation}{section}
\numberwithin{figure}{section}

\begin{document}
\title[Bailey pairs and a conjecture of Hikami]{Bailey pairs, radial limits of $q$-hypergeometric false theta functions, and a conjecture of Hikami}

\author{Jeremy Lovejoy}
\address{CNRS, Universit{\'e} Paris Cit\'e, B\^atiment Sophie Germain, Case Courier 7014,
8 Place Aur\'elie Nemours, 75205 Paris Cedex 13, FRANCE}
\email{lovejoy@math.cnrs.fr}

\author{Rishabh Sarma}
\address{Department of Mathematics, University of Florida, Gainesville, FL, 32611, USA}
\email{rishabh.sarma@ufl.edu}

\date{\today}

\begin{abstract}
In the first part of this paper we prove a conjecture of Hikami on the values of the radial limits of a family of $q$-hypergeometric false theta functions.    Hikami conjectured that the radial limits are obtained by evaluating a truncated version of the series.  He proved a special case of his conjecture by computing the Kashaev invariant of certain torus links in two different ways.   We prove the full conjecture using Bailey pairs.    In the second part of the paper we show how the framework of Bailey pairs leads to further results of this type.   
\end{abstract}

\maketitle
\section{Introduction}
The first main goal of this paper is to prove a conjecture of Hikami \cite{Hik} on the values of the radial limits of certain $q$-hypergeometric series.   To state Hikami's conjecture, first recall the $q$-Pochhammer symbol, defined for $n \geq 0$ by  
\begin{equation*}
(x)_n = (x;q)_n = (1-x)(1-xq) \cdots (1-xq^{n-1}),
\end{equation*}
and the $q$-binomial coefficient (or Gaussian polynomial)
\begin{equation*}
\begin{bmatrix} n \\ k \end{bmatrix} = \begin{bmatrix} n \\ k \end{bmatrix}_q = 
\begin{cases}
\frac{(q;q)_n}{(q;q)_{n-k}(q;q)_k}, & \text{if $0 \leq k \leq n$},\\
0, & \text{otherwise}.
\end{cases}
\end{equation*}
Next, for $m \geq 2$ and $0 \leq a \leq m-2$, define the $q$-hypergeometric series $\widetilde{\Phi}_m^{(a)}(q)$ for $|q| < 1$ by\footnote{Hikami included the prefactor $m$ for convenience.   For consistency, we preserve it here.}
\begin{equation*} 
\begin{aligned}
\widetilde{\Phi}_m^{(a)}(q) = mq^{\frac{(m-a-1)^2}{4m}}\sum_{n_1,\dots, n_{m-1} \geq 0}& (-1)^{n_{m-1}}q^{\binom{n_{m-1} + 1}{2} + n_1^2 + \cdots + n_{m-2}^2 + n_{a+1} + \cdots + n_{m-2}} \\
&\times \prod_{i=1}^{m-2} \begin{bmatrix} n_{i+1} + \delta_{i,a} \\ n_i \end{bmatrix}.
\end{aligned}
\end{equation*}  
Hikami \cite[Corollary 7]{Hik} showed that the series $\widetilde{\Phi}_m^{(a)}(q)$ is the weight $1/2$ false theta function
\begin{equation*}
\widetilde{\Phi}_m^{(a)}(q) = m\sum_{n \geq 0} \chi_{2m}^{(a)}(n)q^{n^2/4m},
\end{equation*}
where $\chi_{2m}^{(a)}(n)$ is the odd periodic function defined by
\begin{equation*}
\chi_{2m}^{(a)}(n) = 
\begin{cases}
1, & \mbox{if $n \equiv m-a-1 \pmod{2m}$},
\\
-1, & \mbox{if $n \equiv m+a+1 \pmod{2m}$}, \\
0, & \mbox{otherwise}.
\end{cases}
\end{equation*}

It is known that such false theta functions have well-defined limiting values as $q$ approaches a root of unity radially from inside the unit disk and that the resulting function is a quantum modular form \cite{Gos-Os}.      Hikami's conjecture states that in the case of $\widetilde{\Phi}_m^{(a)}(q)$ these radial limits are given by evaluating a truncated version of the series.    Namely, define the polynomial $Y_{m,N}^{(a)}(q)$ by 
\begin{equation*}
Y_{m,N}^{(a)}(q) = \sum_{n_1,\dots, n_{m-1} =0}^{N-1} (-1)^{n_{m-1}}q^{\binom{n_{m-1} + 1}{2} + n_1^2 + \cdots + n_{m-2}^2 + n_{a+1} + \cdots + n_{m-2}} \prod_{i=1}^{m-2} \begin{bmatrix} n_{i+1} + \delta_{i,a} \\ n_i \end{bmatrix}.
\end{equation*}
Then we have the following.   Here and throughout the paper we let $\zeta_N = e^{2 \pi i /N}$.
\begin{conjecture}[Hikami \cite{Hik}] \label{Hikamiconjecture}
For any $m \geq 2$ and $0 \leq a \leq m-2$ we have 
\begin{equation*}
\lim_{q\rightarrow\zeta_N}\widetilde{\Phi}_m^{(a)} (q) = \zeta_N^{\frac{(m-a-1)^2}{4m}} Y_{m,N}^{(a)}(\zeta_N).
\end{equation*}
\end{conjecture}   

Hikami proved his conjecture in the case $a=0$ by showing that both sides are essentially the Kashaev invariant of the torus link $T(2,2m)$.    He writes that it is unclear whether these expressions for $a \neq 0$ correspond to any quantum invariant, perhaps hinting that another method is needed to prove the full conjecture.    Here we show that the full conjecture follows from the theory of Bailey pairs.   In fact, it does so quite easily once the Bailey pair framework is set up.  Moreover, it holds for any root of unity.    

\begin{theorem} \label{mainthm}
Let $\zeta_N^M$  be a primitive $N$th root of unity with $N > 0$.    Then for any $m \geq 2$ and $0 \leq a \leq m-2$ we have 
\begin{equation*}
\lim_{q\rightarrow\zeta_N^M}\widetilde{\Phi}_m^{(a)} (q) = {\zeta_N}^{\frac{M(m-a-1)^2}{4m}} Y_{m,N}^{(a)}(\zeta_N^M).
\end{equation*}
\end{theorem}  

In the second part of the paper we show that a similar phenomenon occurs for other families of $q$-hypergeometric false theta functions constructed using Bailey pairs.  That is, the radial limits at roots of unity are obtained by evaluating the truncated series.    See Theorems \ref{ex1theorem}, \ref{ex2theorem}, and \ref{ex3theorem}.

The paper is organized as follows.    In the next section  we prove Hikami's conjecture.    In Section 3 we review the computation of radial limits of false theta functions and present three more results along the lines of Theorem \ref{mainthm}.    We conclude with some remarks on quantum $q$-series identities in the spirit of \cite{Lo1} that arise from our work.



\section{Proof of Theorem \ref{mainthm}}
In this section we prove Theorem \ref{mainthm}.    Before proceeding to the proof we note two facts.   First, at any primitive $N$th root of unity, Hikami \cite[Proposition 10]{Hik} computed the limiting values of $\widetilde{\Phi}_m^{(a)}(q)$.   
\begin{lemma}[Hikami]
For coprime integers $M$ and $N$ with $N > 0$ we have
\begin{equation} \label{values}
\lim_{q\rightarrow\zeta_N^M}\widetilde{\Phi}_m^{(a)} (q) = m \sum_{n = 0}^{mN} \chi_{2m}^{(a)}(n) \left(1- \frac{n}{mN}\right) \zeta_N^{M n^2/ 4m}.
\end{equation} 
\end{lemma}

Second, we need a result from the classical theory of Bailey pairs.    Recall that a pair of sequences $(\alpha_n,\beta_n)$ is called a Bailey pair relative to $(a,q)$ if 
\begin{equation} \label{pairdef}
\beta_n = \sum_{k=0}^n \frac{\alpha_k}{(q)_{n-k}(aq)_{n+k}}.
\end{equation}
The Bailey lemma (see \cite[p. 270]{And} or \cite[Lemma 5.1]{Be-Wa1}) says that if $(\alpha_n,\beta_n)$ is a Bailey pair relative to $(a,q)$ then so is $(\alpha_n',\beta_n')$, where
\begin{equation} \label{alphaprime}
\alpha_n' = \frac{(b)_n(c)_n}{(aq/b)_n(aq/c)_n}(aq/bc)^n \alpha_n
\end{equation}
and
\begin{equation} \label{betaprime}
\beta_n' = \frac{(aq/bc)_n}{(q)_n(aq/b)_n(aq/c)_n}\sum_{k=0}^n \frac{(b)_k(c)_k(q^{-n})_kq^k}{(bcq^{-n}/a)_k}\beta_k.
\end{equation}
Putting this into \eqref{pairdef} with $b,c \to \infty$ and using 
$$
(q)_{n-k} = \frac{(q)_n}{(q^{-n})_k} (-1)^kq^{\binom{k}{2} - nk},
$$
we have the following.
\begin{lemma} \label{keycorollary}
If $(\alpha_n,\beta_n)$ is a Bailey pair relative to $(q,q)$, then 
\begin{equation*}
(q^2)_n\sum_{k=0}^n(q^{-n})_k(-1)^kq^{\binom{k+1}{2} + (n+1)k}\beta_k = \sum_{k=0}^{n} \frac{(q^{-n})_k}{(q^{2+n})_k}(-1)^kq^{\binom{k+1}{2} +(n+1)k} \alpha_k.
\end{equation*}
\end{lemma}

With these two lemmas in hand, we are ready to prove Hikami's conjecture.
\begin{proof}[Proof of Theorem \ref{mainthm}]
We begin by recalling the Bailey pair relative to $(q,q)$ in \cite[Prop. 4.1]{Lov}, 
$$
\alpha_{n}=\frac{(1-q^{(a+1)(2n+1)})(-1)^nq^{\binom{n+1}{2}+(a+1)n^2+(m-a-1)(n^2+n)}}{1-q}
$$
and
$$
\beta_{n} = \beta_{n_m} = \sum_{n_1,\dots , n_{m-1} \geq 0} \frac{q^{n_1^2+\cdots+n_{m-1}^2+n_{a+1}+\cdots+n_{m-1}}}{(q)_{n_m}}\prod_{i=1}^{m-1} \begin{bmatrix} n_{i+1} + \delta_{a,i} \\ n_i \end{bmatrix}.
$$
Here $m \geq 1$ and $0 \leq a \leq m-1$.  When $m=1$ (and $a=0$) the above sum is understood to be empty, and we have
$$
\beta_n = \beta_{n_1}  = \frac{1}{(q)_{n_1}}.
$$

Using this Bailey pair with $m\mapsto m-1$ and substituting into Lemma \ref{keycorollary} with $n=N-1$ we obtain

\begin{equation} \label{pair1inkeycorollary}
\begin{aligned}
(q)_N\sum_{n_1,\dots,n_{m-1} = 0}^{N-1}& \frac{(q^{1-N})_{n_{m-1}}(-1)^{n_{m-1}}q^{\binom{n_{m-1}+1}{2} + Nn_{m-1} +n_1^2+\cdots+n_{m-2}^2 +  n_{a+1}+\cdots+n_{m-2}}}{(q)_{n_{m-1}}} \\
&\times \prod_{i=1}^{m-2} \begin{bmatrix} n_{i+1} + \delta_{a,i} \\ n_i \end{bmatrix} \\
= &\sum_{k=0}^{N-1} \frac{(q^{1-N})_k}{(q^{1+N})_k}(1-q^{(a+1)(2k+1)})q^{mk^2+(m-a-1)k + Nk}.
\end{aligned}
\end{equation}
Note that for $q$ a primitive $N$th root of unity the left hand side vanishes, giving
\begin{align*}
0 &= \sum_{k=0}^{N-1} (1-q^{(a+1)(2k+1)})q^{mk^2+(m-a-1)k} \\ 
&= q^{\frac{-(m-a-1)^2}{4m}}\sum_{k=0}^{2mN}\chi_{2m}^{(a)}(k)q^{k^2/4m},
\end{align*}
where the second equality follows from a short computation involving completing the square.    Dividing both sides of \eqref{pair1inkeycorollary} by $(q)_N$ and taking $\lim_{q \to \zeta_N^M}$ we then have 
\begin{align}
Y_{m,N}^{(a)}(\zeta_N^M) &= \lim_{q \to \zeta_N^M} \frac{1}{(q)_N}\sum_{k=0}^{N-1} \frac{(q^{1-N})_k}{(q^{1+N})_k}(1-q^{(a+1)(2k+1)})q^{mk^2+(m-a-1)k + Nk} \nonumber \\
&= \frac{1}{N}  \lim_{q \to \zeta_N^M} \frac{1}{1-q^N} q^{\frac{-(m-a-1)^2}{4m}}\sum_{k=0}^{2mN}\chi_{2m}^{(a)}(k)q^{k^2/4m} \nonumber \\
&= -\frac{1}{4mN^2} \zeta_N^{\frac{-M(m-a-1)^2}{4m}} \sum_{k=0}^{2mN}k^2 \chi_{2m}^{(a)}(k)\zeta_N^{Mk^2/4m}. \label{finiteeval1}
\end{align}
In the above we have used the fact that for $q$ a primitive $N$th root of unity,
$$
\prod_{i=1}^{N-1} (1-q^i x) = \frac{1-x^N}{1-x},
$$
which gives
$$
(q;q)_{N-1} = N.
$$

Now, to finish the proof we compute 
\begin{align*}
\zeta_N^{\frac{M(m-a-1)^2}{4m}} Y_{m,N}^{(a)}(\zeta_N^M) &=  -\frac{1}{4mN^2}  \sum_{k=0}^{2mN}k^2 \chi_{2m}^{(a)}(k)\zeta_N^{Mk^2/4m} \\
&=  -\frac{1}{4mN^2}  \Bigg( \sum_{k=0}^{mN}k^2 \chi_{2m}^{(a)}(k)\zeta_N^{Mk^2/4m} + \sum_{k=mN}^{2mN}k^2 \chi_{2m}^{(a)}(k)\zeta_n^{Mk^2/4m} \Bigg) \\
&=  -\frac{1}{4mN^2}  \Bigg( \sum_{k=0}^{mN}k^2 \chi_{2m}^{(a)}(k)\zeta_N^{Mk^2/4m} \\
&\hskip1in + \sum_{k=0}^{mN}(2mN-k)^2 \chi_{2m}^{(a)}(2mN-k)\zeta_N^{M(2mN-k)^2/4m} \Bigg) \\
&= -\frac{1}{4mN^2}  \Bigg( \sum_{k=0}^{mN}k^2 \chi_{2m}^{(a)}(k)\zeta_N^{Mk^2/4m} - \sum_{k=0}^{mN}(2mN-k)^2 \chi_{2m}^{(a)}(k)\zeta_N^{Mk^2/4m} \Bigg) \\
&= \sum_{k=0}^{mN} \chi_{2m}^{(a)}(k)\left(m - \frac{k}{N}\right)\zeta_N^{Mk^2/4m} \\
&= \lim_{q\rightarrow\zeta_N^M}\widetilde{\Phi}_m^{(a)}(q),
\end{align*}
the last equality following from \eqref{values}.   This establishes Theorem \ref{mainthm}.

\end{proof}

\section{Further examples}
\subsection{Preliminaries}

In this section we find further examples of families of $q$-hypergeometric false theta functions of weight $1/2$ whose radial limits are given by evaluating a truncated version of the $q$-series.    There are certainly more, but we limit ourselves to applications of some Bailey pairs which have already played a role in the study of $q$-series and weight $3/2$ false theta functions at roots of unity \cite{Lov}.    Here we will need a specialization of the Bailey lemma different from Lemma \ref{keycorollary}.     Namely, if we take $a=b = q$ and $c \to \infty$ in \eqref{alphaprime} and \eqref{betaprime} and then use the definition \eqref{pairdef} with $n \to \infty$ we have the following, which is well-known.
\begin{lemma} \label{keylemma2}
If $(\alpha_n,\beta_n)$ is a Bailey pair relative to $(q,q)$, then 
\begin{equation*}
\sum_{n \geq 0}(q)_n(-1)^nq^{\binom{n+1}{2}}\beta_n = (1-q)\sum_{n \geq 0} (-1)^nq^{\binom{n+1}{2}}\alpha_n.
\end{equation*} 
\end{lemma} 

We also require a formula which allows us to compute the radial limits of false theta functions as we approach a root of unity.    We cite this more or less verbatim from Hikami \cite[Proposition 9]{Hik}.
\begin{lemma} \label{valuesgeneral}
Let $C_f(n)$ be a periodic function with mean value $0$ and modulus $f$.   Then as $t \searrow 0$ we have the asymptotic expansion
\begin{equation*}
\sum_{n \geq 1} C_f(n)e^{-n^2t} \sim \sum_{k \geq 0} L(-2k,C_f)\frac{(-t)^k}{k!},
\end{equation*}
where
\begin{equation*}
L(-k,C_f) = - \frac{f^k}{k+1}\sum_{n=1}^f C_f(n)B_{k+1}\left(\frac{n}{f}\right),
\end{equation*}
with $B_k(x)$ being the $k$th Bernouilli polynomial.    
\end{lemma}

In each of the following subsections we first use a known Bailey pair together with Lemma \ref{keylemma2} to produce a family of $q$-hypergeometric false theta functions.   These are not all necessarily new, but we include the derivations for completeness.    Next we use Lemma \ref{valuesgeneral} to compute the radial limits of the false theta functions.    Finally, we use the same Bailey pair in Lemma \ref{keycorollary} to produce a truncated version of the $q$-series whose values at roots of unity coincide with the radial limits of the infinite series.  

\subsection{Example 1} For $m \geq 2$ define 
\begin{equation*}
\begin{aligned}
\widetilde{\Psi}_m(q) = \frac{2m-1}{2}q^{\frac{(2m-3)^2}{8(2m-1)}}\sum_{n_1,\dots,n_{m-1} \geq 0}& \frac{(-1)^{n_{m-1}}q^{\binom{n_{m-1}+1}{2}+n_1^2+\cdots+n_{m-2}^2+n_{1}+\cdots+n_{m-2}}}{(-q)_{n_1}} \\
&\times \prod_{i=1}^{m-2} \begin{bmatrix} n_{i+1} \\ n_i \end{bmatrix}
\end{aligned}
\end{equation*}
and its truncated version 
\begin{equation*}
Z_{m,N}(q) = \sum_{n_1, \dots, n_{m-1} = 0}^{N-1} \frac{(-1)^{n_{m-1}}q^{\binom{n_{m-1}+1}{2}+n_1^2+\cdots+n_{m-2}^2+n_{1}+\cdots+n_{m-2}}}{(-q)_{n_1}}\prod_{i=1}^{m-2} \begin{bmatrix} n_{i+1} \\ n_i \end{bmatrix}.
\end{equation*}
Our goal is to prove that the radial limits of $\widetilde{\Psi}_m(q)$ are well-defined at odd roots of unity and that the values are essentially given by $Z_{m,N}(q)$.   We restrict to odd roots of unity to avoid poles arising from the term $(-q)_{n_1}$ in the denominator of $Z_{m,N}(q)$.  
\begin{theorem} \label{ex1theorem}
Let $\zeta_N^M$ be a primitive odd $N$th root of unity.    Then 
\begin{equation*}
\lim_{q\rightarrow\zeta_N^M}\widetilde{\Psi}_m(q) = \zeta_N^{\frac{M(2m-3)^2}{8(2m-1)}} Z_{m,N}(\zeta_N^M).
\end{equation*}
\end{theorem}

This result will follow from the next three propositions.   We first prove that $\widetilde{\Psi}_m(q)$ is a false theta function.
\begin{proposition} \label{ex1prop1}
We have the identity
\begin{equation*}
\widetilde{\Psi}_m(q) = \frac{2m-1}{2} \sum\limits_{k=0}^{\infty}\chi_{4m-2}(k)\,q^{\frac{k^2}{8(2m-1)}} ,
\end{equation*}
where $\chi_{4m-2}(k)$ is the odd periodic function defined by
\begin{equation*}
\chi_{4m-2}(k)=\begin{cases}
1, & \mbox{if $k \equiv 2m-3 \pmod{4m-2}$},
\\
-1, & \mbox{if $k \equiv 2m+1 \pmod{4m-2}$}, \\
0, & \mbox{otherwise}.
\end{cases}
\end{equation*}
\end{proposition}

\begin{proof}
We begin with the Bailey pair relative to $(q,q)$ (see \cite[Proposition 5.1]{Lov} or \cite[p. 373]{Wa1}), 
\begin{equation} \label{alphaex1}
\alpha_n = \frac{1-q^{2n+1}}{1-q}(-1)^nq^{mn^2 + (m-1)n}
\end{equation}
and
\begin{equation} \label{betaex1}
\beta_n = \beta_{n_{m}} = \sum_{n_1, \dots,n_{m-1} \geq 0} \frac{q^{n_1^2+\cdots+n_{m-1}^2+n_{1}+\cdots+n_{m-1}}}{(q)_{n_{m}}(-q)_{n_1}}   \prod_{i=1}^{m-1} \begin{bmatrix} n_{i+1} \\ n_i \end{bmatrix}.
\end{equation}
Here $m \geq 1$ and when $m=1$ we have
$$
\beta_n = \beta_{n_1} = \frac{1}{(q^2;q^2)_{n_1}}.
$$
Using \eqref{alphaex1} and \eqref{betaex1} with $m\mapsto m-1$ in Lemma \ref{keylemma2} we obtain
\begin{align*}
\widetilde{\Psi}_m(q) &= \frac{2m-1}{2}q^{\frac{(2m-3)^2}{8(2m-1)}}\sum_{n \geq 0} (1-q^{2n+1})q^{\binom{n+1}{2} + (m-1)n^2+ (m-2)n}
\end{align*}
and the result then follows after completing the square on the right-hand side.
\end{proof}

Next we employ Lemma \ref{valuesgeneral} to calculate the radial limits of $\widetilde{\Psi}_m(q)$ as $q$ approaches an odd root of unity.   

\begin{proposition} \label{ex1prop2}
For coprime integers $M$ and $N$ with $N$ odd and positive we have
\begin{equation*}
\lim_{q\rightarrow\zeta_N^M}\widetilde{\Psi}_m(q) = \dfrac{-1}{8(2m-1)N^2}\sum_{k=0}^{(4m-2)N}k^2\chi_{4m-2}(k)\zeta_N^{\frac{k^2}{8(2m-1)}M}.
\end{equation*}
\end{proposition}

\begin{proof}
Regarding $\widetilde{\Psi}_m$ as a function of $\tau$ with $q = e^{2 \pi i \tau}$, we have 
\begin{align*}
\widetilde{\Psi}_m\left(\frac{M}{N}+i\frac{t}{2\pi}\right)&=\frac{2m-1}{2}\sum_{k=0}^{\infty}\chi_{4m-2}(k)e^{2\pi i \left(\frac{M}{N}+i\frac{t}{2\pi}\right) \frac{k^2}{8(2m-1)}}
\\
&=\frac{2m-1}{2}\sum_{k=0}^{\infty}C_{(4m-2)N}(k)e^{ \frac{-k^2}{8(2m-1)}t},
\end{align*}
where 
$$
C_{(4m-2)N}(k)=\chi_{4m-2}(k)e^{\frac{Mk^2}{2(4m-2)N}\pi i}.
$$
We have  
\begin{align*}
C_{(4m-2)N}(k + (4m-2)N) &= \chi_{4m-2}(k + (4m-2)N)e^{\frac{M\pi i(k^2+ 2kN(4m-2) + N^2(4m-2)^2)}{2(4m-2)N}} \\
&= \chi_{4m-2}(k) e^{\frac{M\pi i k^2}{2(4m-2)N}}e^{M\pi i(k + N(2m-1))},
\end{align*}
and using the fact that $N$ is odd together with the definition of $\chi_{4m-2}(k)$ this gives 
\begin{equation*} \label{Cprop1ex1}
C_{(4m-2)N}(k+(4m-2)N)=C_{(4m-2)N}(k).
\end{equation*}
Similarly, we have
\begin{equation} \label{Cprop2ex1}
C_{(4m-2)N}((4m-2)N-k)=-C_{(4m-2)N}(k).
\end{equation}
Note that by \eqref{Cprop2ex1} we have
\begin{equation} \label{Cprop3ex1}
\sum_{k = 1}^{(4m-2)N} C_{(4m-2)N}(k) = \sum_{k = 0}^{(4m-2)N} C_{(4m-2)N}(k) = 0.
\end{equation}

Now using Lemma \ref{valuesgeneral} we have the asymptotic expansion as $t\searrow0$, $$\widetilde{\Psi}_m\left(\frac{M}{N}+i\frac{t}{2\pi}\right) \sim \frac{2m-1}{2} \sum_{k=0}^{\infty}\dfrac{L(-2k,C_{(4m-2)N})}{k!}\left(-\dfrac{t}{8(2m-1)}\right)^k.$$
Using this together with $B_1(x) = x - \frac{1}{2}$ we compute the limiting value as follows,
\begin{align*}
\lim_{\tau \downarrow \frac{M}{N}} \widetilde{\Psi}_m (\tau)&=\frac{2m-1}{2}L(0,C_{(4m-2)N})
\\
&= -\frac{2m-1}{2}\sum_{k=1}^{(4m-2)N}C_{(4m-2)N}(k)B_1\left(\dfrac{k}{(4m-2)N}\right)
\\
&=- \frac{2m-1}{2}\sum_{k=1}^{(4m-2)N}C_{(4m-2)N}(k)\left(\frac{k}{(4m-2)N}\right)   \ \ \ \text{(by \eqref{Cprop3ex1})}
\\
&=- \sum_{k=0}^{(4m-2)N} C_{(4m-2)N}(k)\left(\dfrac{k}{4N}\right)
\\
&=- \sum_{k=0}^{(2m-1)N} C_{(4m-2)N}(k)\left(\dfrac{k}{4N}\right) 
\\
&\ \ \ \ \ \ \ \ - \sum_{k=0}^{(2m-1)N} C_{(4m-2)N}((4m-2)N - k)\left(\dfrac{(4m-2)N - k}{4N}\right)
\\
&=\sum_{k=0}^{(2m-1)N}C_{(4m-2)N}(k)\left(\frac{2m-1}{2}-\dfrac{k}{2N}\right) \ \ \ \text{(by \eqref{Cprop2ex1})}.
\end{align*}
Using the fact that
$$
\frac{2m-1}{2}-\dfrac{k}{2N}  = \frac{-1}{8(2m-1)N^2}\left(k^2 - ((4m-2)N - k)^2\right)
$$
we then obtain
\begin{align*}
\lim_{\tau \downarrow \frac{M}{N}} \widetilde{\Psi}_m (\tau) =\dfrac{-1}{8(2m-1)N^2}\Bigg(&\sum_{k=0}^{(2m-1)N}k^2C_{(4m-2)N}(k) \\
&-\sum_{k=0}^{(2m-1)N}((4m-2)N-k)^2C_{(4m-2)N}(k)\Bigg)
\\
=\dfrac{-1}{8(2m-1)N^2}\Bigg(&\sum_{k=0}^{(2m-1)N}k^2C_{(4m-2)N}(k) \\
&+\sum_{k=0}^{(2m-1)N}((4m-2)N-k)^2 C_{(4m-2)N}((4m-2)N - k)\Bigg)
\\
=\dfrac{-1}{8(2m-1)N^2}\Bigg(&\sum_{k=0}^{(2m-1)N}k^2 C_{(4m-2)N}(k) + \sum_{k=(2m-1)N}^{(4m-2)N}k^2 C_{(4m-2)N}(k)\Bigg)
\\
=\dfrac{-1}{8(2m-1)N^2}&\sum_{k=0}^{(4m-2)N}k^2\chi_{4m-2}(k)e^{\frac{k^2}{2(4m-2)N}M\pi i},
\end{align*}

which gives the result.
\end{proof}

Now we determine the value of $Z_{m,N}(q)$ at primitive odd $N$th roots of unity.

\begin{proposition} \label{ex1prop3}
For $q = \zeta_N^M$, a primitive odd $N$-th root of unity, we have
\begin{equation} \label{finiteeval2}
Z_{m,N}(q) =\frac{-1}{8(2m-1)N^2}\zeta_N^{\frac{-M(2m-3)^2}{8(2m-1)}}\sum\limits_{k=0}^{(4m-2)N}k^2\chi_{4m-2}(k)\,\zeta_N^{\frac{Mk^2}{8(2m-1)}}.
\end{equation}
\end{proposition}

\begin{proof}
Inserting the Bailey pair from \eqref{alphaex1} and \eqref{betaex1} into Lemma \ref{keycorollary} with $m\mapsto m-1$ and $n=N-1$, we obtain
\\
\begin{align*}
\sum_{n_1, \dots , n_{m-1} = 0}^{N-1}& \frac{(q^{1-N})_{n_{m-1}}(-1)^{n_{m-1}}q^{Nn_{m-1}+\binom{n_{m-1}+1}{2}+n_1^2+\cdots+n_{m-2}^2+n_{1}+\cdots+n_{m-2}}}{(q)_{n_{m-1}}(-q)_{n_1}}\prod_{i=1}^{m-2} \begin{bmatrix} n_{i+1} \\ n_i \end{bmatrix}
\\
&=\sum_{k=0}^{N-1}\frac{(q^{1-N})_k(1-q^{2k+1})q^{-\binom{k}{2}+(N-1)k+mk^2+(m-1)k}}{(q)_{N+k}}.
\end{align*}
Then taking $q$ to be a primitive odd $N$-th root of unity $\zeta_N^M$, we get
\begin{align*}
Z_{m,N}(q) 
&=\lim_{q \rightarrow \zeta_N^M}\sum_{k=0}^{N-1}\frac{(q^{1-N})_k(1-q^{2k+1})\,q^{-\binom{k}{2} +(N-1)k+mk^2+(m-1)k}}{(q)_{N-1}(1-q^N)(q^{N+1})_k}
\\
&=\frac{1}{N}\lim_{q \rightarrow \zeta_N^M}\frac{\sum\limits_{k=0}^{N-1}(1-q^{2k+1})\,q^{-\binom{k}{2}+mk^2+mk-2k}}{(1-q^N)}
\\
&=\frac{1}{N}\lim_{q \rightarrow \zeta_N^M}\frac{q^{\frac{-(2m-3)^2}{8(2m-1)}}\sum\limits_{k=0}^{(4m-2)N}\chi_{4m-2}(k)\,q^{\frac{k^2}{8(2m-1)}}}{(1-q^N)}
\\
&=\frac{-1}{8(2m-1)N^2}\zeta_N^{\frac{-M(2m-3)^2}{8(2m-1)}}\sum\limits_{k=0}^{(4m-2)N}k^2\chi_{4m-2}(k)\,\zeta_N^{\frac{Mk^2}{8(2m-1)}}.
\end{align*}
This completes the proof.
\end{proof}

Comparing Propositions \ref{ex1prop2} and \ref{ex1prop3} gives Theorem \ref{ex1theorem}.

\subsection{Example 2} We follow the same steps as in the previous example, this time with the function
\begin{equation*}
\begin{aligned}
\widetilde{\Gamma}_m(q) = (m-1)q^{\frac{(2m-3)^2}{8(m-1)}} \sum_{n_1, \dots, n_{m-1} \geq 0} & \frac{(-1)^{n_{m-1}}q^{2\binom{n_{m-1}+1}{2}+2n_1^2+ 2n_1 + \cdots+2n_{m-2}^2+2n_{m-2}}(q;q^2)_{n_1}}{(-q)_{2n_1+1}} \\
&\times \prod_{i=1}^{m-2} \begin{bmatrix} n_{i+1} \\ n_i \end{bmatrix}_{q^2}
\end{aligned}
\end{equation*}
and the truncated version
\begin{equation*}
U_{m,N}(q) = \sum_{n_1, \dots , n_{m-1} = 0}^{N-1} \frac{(-1)^{n_{m-1}}q^{2\binom{n_{m-1}+1}{2}+2n_1^2+ 2n_1 + \cdots+2n_{m-2}^2 +2n_{m-2}}(q;q^2)_{n_1}}{(-q)_{2n_1+1}}\prod_{i=1}^{m-2} \begin{bmatrix} n_{i+1} \\ n_i \end{bmatrix}_{q^2}.
\end{equation*}

As usual, these are valid for $m \geq 2$.    We will show that $\widetilde{\Gamma}_m(q)$ is a false theta function and that its limiting values at odd roots of unity are computed using $U_{m,N}(q)$.

\begin{theorem} \label{ex2theorem}
Let $\zeta_N^M$ be a primitive odd $N$th root of unity.    Then 
\begin{equation*}
\lim_{q\rightarrow\zeta_N^M}\widetilde{\Gamma}_m(q) = \zeta_N^{\frac{M(2m-3)^2}{8(m-1)}} U_{m,N}(\zeta_N^M).
\end{equation*}
\end{theorem}

This theorem will follow from the next three propositions.   We first show that $\widetilde{\Gamma}_m(q)$ is a false theta function.

\begin{proposition}
We have the identity
\begin{equation*}
\widetilde{\Gamma}_m(q)= (m-1)\sum\limits_{k=0}^{\infty}\chi_{4(m-1)}(k)\,q^{\frac{k^2}{8(m-1)}},
\end{equation*}
where
$$\chi_{4(m-1)}(k)=\begin{cases}
1, & \mbox{if $k \equiv 2m-3 \pmod{4(m-1)}$},
\\
-1, & \mbox{if $k \equiv 2m-1 \pmod{4(m-1)}$}, \\
0, & \mbox{otherwise}.
\end{cases}$$
\end{proposition}

\begin{proof}
We use the Bailey pair relative to $(q^2,q^2)$ \cite[Proposition 5.3]{Lov},
\begin{equation} \label{alphaex2}
\alpha_n = \frac{(-1)^nq^{(2m-1)n^2 + (2m-2)n}(1-q^{2n+1})}{1-q^2}
\end{equation}
and
\begin{equation} \label{betaex2}
\beta_n = \beta_{n_m} = \sum_{n_1, \dots, n_{m-1} \geq 0} \frac{q^{2n_1^2+2n_1 + \cdots + 2n_{m-1}^2 + 2n_{m-1}}(q;q^2)_{n_1}}{(q^2;q^2)_{n_m}(-q)_{2n_1 + 1}} \prod_{i=1}^{m-1} \begin{bmatrix} n_{i+1} \\ n_i \end{bmatrix}_{q^2}.
\end{equation}
Here $m \geq 1$ and when $m=1$ we have 
$$
\beta_n = \beta_{n_1} = \frac{(q;q^2)_{n_1}}{(q^2;q^2)_{n_1}(-q)_{2n_1+1}}.
$$
Inserting this Bailey pair with $m\mapsto m-1$ into Lemma \ref{keylemma2} (remembering to replace $q$ by $q^2$ throughout) we obtain
\begin{equation*}
\widetilde{\Gamma}_m(q) = (m-1)q^{\frac{(2m-3)^2}{8(m-1)}} \sum_{n \geq 0} q^{(2m-2)n^2 + (2m-3)n}(1-q^{2n+1}),
\end{equation*}
and the result then follows after completing the square on the right-hand side.
\end{proof}

Next we use Lemma \ref{valuesgeneral} to compute $\widetilde{\Gamma}_m(q)$ as $q$ approaches an odd root of unity.

\begin{proposition} \label{ex2prop2}
For coprime integers $M$ and $N$ with $N$ odd and positive we have
\begin{equation*}
\lim_{q\rightarrow\zeta_N^M}\widetilde{\Gamma}_m (q)= \dfrac{-1}{16(m-1)N^2}\sum_{k=0}^{(4m-4)N}k^2\chi_{4(m-1)}(k)\zeta_N^{\frac{k^2}{8(m-1)}M}.
\end{equation*}
\end{proposition}

\begin{proof}
Regarding $\widetilde{\Gamma}_m$ as a function of $\tau$ with $q = e^{2 \pi i \tau}$ we have

\begin{align*}
\widetilde{\Gamma}_m\left(\frac{M}{N}+i\frac{t}{2\pi}\right)&=(m-1)\sum_{k=0}^{\infty}\chi_{4(m-1)}(k)e^{2\pi i \left(\frac{M}{N}+i\frac{t}{2\pi}\right) \frac{k^2}{8(m-1)}}
\\
&=(m-1)\sum_{k=0}^{\infty}C_{4(m-1)N}(k)e^{ \frac{-k^2}{8(m-1)}t},
\end{align*}
where 
$$
C_{4(m-1)N}(k)=\chi_{4(m-1)}(k)e^{\frac{Mk^2}{4(m-1)N}\pi i}.
$$
We note that 
\begin{equation*} \label{Cprop1ex2}
C_{4(m-1)N}(k+4(m-1)N)=C_{4(m-1)N}(k)
\end{equation*} 
and 
\begin{equation} \label{Cprop2ex2}
C_{4(m-1)N}(4(m-1)N-k)=-C_{4(m-1)N}(k).
\end{equation}
Then using Lemma \ref{valuesgeneral} we have, asymptotically as $t \searrow 0$,

$$\widetilde{\Gamma}_m\left(\frac{M}{N}+i\frac{t}{2\pi}\right) \sim (m-1)\sum_{k=0}^{\infty}\dfrac{L(-2k,C_{4(m-1)N})}{k!}\left(-\dfrac{t}{8(m-1)}\right)^k.$$
Using this and keeping in mind \eqref{Cprop2ex2} we compute the limiting value
\begin{small}
\begin{align*}
\lim_{\tau \downarrow \frac{M}{N}} \widetilde{\Gamma}_m(\tau)&=(m-1)L(0,C_{4(m-1)N})
\\
&= - (m-1)\sum_{k=1}^{4(m-1)N}C_{4(m-1)N}(k)B_1\left(\dfrac{k}{4(m-1)N}\right)
\\
&=(m-1)\sum_{k=1}^{4(m-1)N}C_{4(m-1)N}(k)\left(\frac{1}{2} - \dfrac{k}{4(m-1)N}\right)
\\
&=(m-1)\sum_{k=0}^{2(m-1)N}C_{4(m-1)N}(k)\left(1-\dfrac{k}{2(m-1)N}\right)
\\
&=\sum_{k=0}^{2(m-1)N}C_{4(m-1)N}(k)\left((m-1)-\dfrac{k}{2N}\right)
\\
&=\dfrac{-1}{16(m-1)N^2}\left(\sum_{k=0}^{2(m-1)N}k^2C_{4(m-1)N}(k)-\sum_{k=0}^{2(m-1)N}(4(m-1)N-k)^2C_{4(m-1)N}(k)\right)
\\
&=\dfrac{-1}{16(m-1)N^2}\Bigg(\sum_{k=0}^{2(m-1)N}k^2C_{4(m-1)N}(k) \\
&\ \ \ \ \ \ \ \ \ \ \ \ \ \ \ \ \ \ \ \ \ \ \ \ \ \ \ \ \ \ +\sum_{k=0}^{2(m-1)N}(4(m-1)N-k)^2 C_{4(m-1)N}(4(m-1)N - k)\Bigg)
\\
&=\dfrac{-1}{16(m-1)N^2}\left(\sum_{k=0}^{2(m-1)N}k^2C_{4(m-1)N}(k)+\sum_{k=2(m-1)}^{4(m-1)N}k^2C_{4(m-1)N}(k)\right)
\\
&=\dfrac{-1}{16(m-1)N^2}\sum_{k=0}^{(4m-4)N}k^2\chi_{4(m-1)}(k)e^{\frac{k^2}{4(m-1)N}M\pi i},
\end{align*}
\end{small}
and the result follows.
\end{proof}

Finally, we compute the values of the rational function $U_{m,N}(q)$ when $q$ is a primitive odd $N$th root of unity.

\begin{proposition} \label{ex2prop3}
 For $q = \zeta_N^M$, a primitive odd $N$-th root of unity, we have
\begin{equation} \label{finiteeval3}
U_{m,N}(q) =\frac{-1}{16(m-1)N^2}\zeta_N^{\frac{-M(2m-3)^2}{8(m-1)}}\sum\limits_{k=0}^{(4m-4)N}k^2\chi_{4(m-1)}(k)\,\zeta_N^{\frac{Mk^2}{8(m-1)}}.
\end{equation}
\end{proposition}

\begin{proof}
Inserting the Bailey pair \eqref{alphaex2} and \eqref{betaex2} into Lemma \ref{keycorollary} with $m \mapsto m-1$ and $n = N-1$ we find 
\\
\begin{align*}
U_{m,N}(\zeta_N^M) &= 
\lim_{q \rightarrow \zeta_N^M}\sum_{k=0}^{N-1}\frac{(q^{2-2N};q^2)_k(1-q^{2k+1})\,q^{2(m-1)k^2+(2m-3)k+2Nk}}{(q^2;q^2)_{N-1}(1-q^{2N})(q^{2N+2};q^2)_k}
\\
&=\frac{1}{N}\lim_{q \rightarrow \zeta_N^M}\frac{\sum\limits_{k=0}^{N-1}(1-q^{2k+1})\,q^{2(m-1)k^2+(2m-3)k}}{1-q^{2N}}
\\
&=\frac{1}{N}\lim_{q \rightarrow \zeta_N^M}\frac{q^{\frac{-(2m-3)^2}{8(m-1)}}\sum\limits_{k=0}^{(4m-4)N}\chi_{4(m-1)}(k)\,q^{\frac{k^2}{8(m-1)}}}{1-q^{2N}}
\\
&=\frac{-1}{16(m-1)N^2}\zeta_N^{\frac{-M(2m-3)^2}{8(m-1)}}\sum\limits_{k=0}^{(4m-4)N}k^2\chi_{4(m-1)}(k)\,\zeta_N^{\frac{Mk^2}{8(m-1)}}.
\end{align*}
\end{proof}

Combining Propositions \ref{ex2prop2} and \ref{ex2prop3} gives Theorem \ref{ex2theorem}.

\subsection{Example 3}
For $m \geq 2$ and $0 \leq a \leq m-2$ define $\widetilde{\Lambda}_m^{(a)}(q)$ by
\begin{equation*}
\begin{aligned}
\widetilde{\Lambda}_m^{(a)}(q) = \frac{2m-1}{2} q^{\frac{(m-a-1)^2}{(2m-1)}} \sum_{n_1, \dots, n_{m-1} \geq 0}& \frac{(-1)^{n_{m-1}}q^{2\binom{n_{m-1}+1}{2}+2n_1^2+\cdots+2n_{m-2}^2+2n_{a+1}+\cdots+ 2n_{m-2}}}{(-q;q^2)_{n_1+\delta_{a,0}}} \\
&\times \prod_{i=1}^{m-2} \begin{bmatrix} n_{i+1}+\delta_{i,a} \\ n_i \end{bmatrix}_{q^2}
\end{aligned}
\end{equation*}
along with its truncated counterpart
\begin{equation*}
V_{m,N}^{(a)}(q) =  \sum_{n_1, \dots, n_{m-1} =  0}^{N-1} \frac{(-1)^{n_{m-1}}q^{2\binom{n_{m-1}+1}{2}+2n_1^2+\cdots+2n_{m-2}^2+2n_{a+1}+\cdots+ 2n_{m-2}}}{(-q;q^2)_{n_1+\delta_{a,0}}}\prod_{i=1}^{m-2} \begin{bmatrix} n_{i+1}+\delta_{i,a} \\ n_i \end{bmatrix}_{q^2}.
\end{equation*}

We will show the following.
\begin{theorem} \label{ex3theorem}
Let $\zeta_N^M$ be a primitive odd $N$th root of unity.    Then 
\begin{equation*}
\lim_{q\rightarrow\zeta_N^M}\widetilde{\Lambda}_m^{(a)}(q) = \zeta_N^{\frac{M(m-a-1)^2}{(2m-1)}} V_{m,N}^{(a)}(\zeta_N^M).
\end{equation*}
\end{theorem}

The steps should be familiar.    We prove a series of three propositions, beginning with the fact that $\widetilde{\Lambda}_m^{(a)}(q)$ is a false theta function. 

\begin{proposition}
We have the identity
\begin{equation*}
\widetilde{\Lambda}_m^{(a)}(q) =\frac{2m-1}{2}\sum\limits_{k=0}^{\infty}\chi_{2(2m-1)}^{(a)}(k)\,q^{\frac{k^2}{4(2m-1)}},
\end{equation*}
where
$$
\chi_{2(2m-1)}^{(a)}(k)=
\begin{cases}
1, & \mbox{if $k \equiv 2(m-a-1) \pmod{2(2m-1)}$},
\\
-1, & \mbox{if $k \equiv 2(m+a) \pmod{2(2m-1)}$}, \\
0 & \mbox{otherwise}.
\end{cases}
$$
\end{proposition}

\begin{proof}
We begin with the Bailey pair relative to $(q^2,q^2)$ \cite[Proposition 5.5]{Lov}\footnote{We note that there is a typo in Proposition 5.5 of \cite{Lov}.   The term $(1+x^{2a+1}q^{(2a+1)(2n+1)})$ should be $(1-x^{2a+1}q^{(2a+1)(2n+1)})$.},
\begin{equation} \label{alphaex3}
\alpha_n = \frac{1}{1-q^2}(-1)^nq^{2(m-a-1)(n^2+n) + 2(a+1)n^2 + n}(1-q^{(2a+1)(2n+1)})
\end{equation}
and
\begin{equation} \label{betaex3}
\beta_n = \beta_{n_m} = \sum_{n_1, \dots, n_{m-1} \geq 0} \frac{q^{2n_1^2+\cdots+2n_{m-1}^2+2n_{a+1}+\cdots+2n_{m-1}}}{(q^2;q^2)_{n_m}(-q;q^2)_{n_1+\delta_{a,0}}}\prod_{i=1}^{m-1} \begin{bmatrix} n_{i+1}+\delta_{i,a} \\ n_i \end{bmatrix}_{q^2}.
\end{equation}
Here $m \geq 1$ and $0 \leq a \leq m-1$.   When $m=1$ (and $a=0$) we have
$$
\beta_n = \beta_{n_1} = \frac{1}{(q^2;q^2)_{n_1}(-q;q^2)_{n_1+1}}.
$$
Using the case $m\mapsto m-1$ of this Bailey pair in Lemma \ref{keylemma2} (with $q=q^2$) we have
\begin{equation*}
\widetilde{\Lambda}_m^{(a)}(q) = \frac{2m-1}{2} q^{\frac{(m-a-1)^2}{(2m-1)}} \sum_{n=0}^{\infty}(1-q^{(2a+1)(2n+1)})\,q^{(2m-1)n^2+2(m-a-1)n},
\end{equation*}
and this gives the result after rewriting the right-hand side.
\end{proof}

Now we compute the of limiting values of $\Lambda_m^{(a)}(q)$ as $q$ approaches a root of unity.   

\begin{proposition} \label{ex3prop2}
For coprime integers $M$ and $N$ with $N$ positive we have
\begin{equation*}
\lim_{q\rightarrow\zeta_N^M}\widetilde{\Lambda}_m^{(a)}(q) = \frac{-1}{8(2m-1)N^2}\sum\limits_{k=0}^{(4m-2)N}k^2\chi_{2(2m-1)}^{(a)}(k)\,\zeta_N^{\frac{Mk^2}{4(2m-1)}}.
\end{equation*}
\end{proposition}

\begin{proof}
Regarding $\widetilde{\Lambda}_m^{(a)}$ as a function of $\tau$ with $q = e^{2 \pi i \tau}$, we have

\begin{align*}
\widetilde{\Lambda}_m^{(a)}\left(\frac{M}{N}+i\frac{t}{2\pi}\right)&=\frac{2m-1}{2}\sum_{k=0}^{\infty}\chi_{2(2m-1)}^{(a)}(k)e^{2\pi i \left(\frac{M}{N}+i\frac{t}{2\pi}\right) \frac{k^2}{4(2m-1)}}
\\
&=\frac{2m-1}{2}\sum_{k=0}^{\infty}C_{2(2m-1)N}(k)e^{ \frac{-k^2}{4(2m-1)}t},
\end{align*}
where 
$$
C_{2(2m-1)N}(k)=\chi_{2(2m-1)}^{(a)}(k)e^{\frac{Mk^2}{2(2m-1)N}\pi i}.
$$
We note that 
\begin{equation*} \label{Cprop1ex3}
C_{2(2m-1)N}(k+2(2m-1)N)=C_{2(2m-1)N}(k)
\end{equation*}
and 
\begin{equation} \label{Cprop2ex3}
C_{2(2m-1)N}(2(2m-1)N-k)=-C_{2(2m-1)N}(k).
\end{equation}
Then using Lemma \ref{valuesgeneral} we have the asymptotic expansion as $t \searrow 0$,

$$\widetilde{\Lambda}_m^{(a)}\left(\frac{M}{N}+i\frac{t}{2\pi}\right) \sim \sum_{k=0}^{\infty}\dfrac{L(-2k,C_{2(2m-1)N})}{k!}\left(-\dfrac{t}{4(2m-1)}\right)^k.$$
Using this and keeping in mind \eqref{Cprop2ex3} we compute the limiting value
\begin{align*}
\lim_{\tau \downarrow \frac{M}{N}} \widetilde{\Lambda}_m^{(a)}(\tau)&=\frac{2m-1}{2}L(0,C_{2(2m-1)N})
\\
&=-\frac{2m-1}{2}\sum_{k=1}^{2(2m-1)N}C_{2(2m-1)N}(k)B_1\left(\dfrac{k}{2(2m-1)N}\right)
\\
&=\frac{2m-1}{2}\sum_{k=1}^{2(2m-1)N}C_{2(2m-1)N}(k)\left(\frac{1}{2} - \dfrac{k}{2(2m-1)N}\right)
\\
&=\frac{2m-1}{2}\sum_{k=0}^{(2m-1)N}C_{2(2m-1)N}(k)\left(1-\dfrac{k}{(2m-1)N}\right)
\\
&=\sum_{k=0}^{(2m-1)N}C_{2(2m-1)N}(k)\left(\frac{2m-1}{2}-\dfrac{k}{2N}\right)
\\
&=\dfrac{-1}{8(2m-1)N^2}\Bigg(\sum_{k=0}^{(2m-1)N}k^2C_{2(2m-1)N}(k) \\
&\ \ \ \ \ \ \ \ \ \ \ \ \ \ \ \ \ \ \ \ \ \ \ \ \ \ \ \ -\sum_{k=0}^{(2m-1)N}(2(2m-1)N-k)^2C_{2(2m-1)N}(k)\Bigg)
\\
&=\dfrac{-1}{8(2m-1)N^2}\Bigg(\sum_{k=0}^{(2m-1)N}k^2C_{2(2m-1)N}(k) \\
&\ \ \ \ \ \ \ \ \ \ \ \ \ \ \ \ \ \ \ \ \ \ \ \ \ \ \ \ +\sum_{k=0}^{(2m-1)N}(2(2m-1)N-k)^2C_{2(2m-1)N}(2(2m-1)N - k)\Bigg)
\\
&=\dfrac{-1}{8(2m-1)N^2}\left(\sum_{k=0}^{(2m-1)N}k^2C_{2(2m-1)N}(k)+\sum_{k=(2m-1)}^{2(2m-1)N}k^2C_{2(2m-1)N}(k)\right)
\\
&=\frac{-1}{8(2m-1)N^2}\sum\limits_{k=0}^{(4m-2)N}k^2\chi_{2(2m-1)}^{(a)}(k)\,e^{\frac{k^2}{2(2m-1)N}M \pi i},
\end{align*}
which gives the result.
\end{proof}

\begin{proposition}  \label{ex3prop3}
For $q = \zeta_N^M$, a primitive odd $N$-th root of unity, we have
\begin{equation} \label{finiteeval4}
V_{m,N}^{(a)}(q) =\frac{-1}{8(2m-1)N^2}\zeta_N^{\frac{-M(m-a-1)^2}{(2m-1)}}\sum\limits_{k=0}^{(4m-2)N}k^2\chi_{2(2m-1)}^{(a)}(k)\,\zeta_N^{\frac{Mk^2}{4(2m-1)}}.
\end{equation}
\end{proposition}

\begin{proof}
Inserting the Bailey pair from \eqref{alphaex3} and \eqref{betaex3} with $m\mapsto m-1$ into Lemma \ref{keycorollary} with $n=N-1$, we obtain
\begin{align*}
V_{m,N}^{(a)}&(\zeta_N^M) \\
&=\lim_{q \to \zeta_M^N} \sum_{k=0}^{N-1}\frac{(q^{2-2N};q^2)_k(1-q^{(2a+1)(2k+1)})q^{2(a+1)k^2+k+2(m-a-1)(k^2+k)+2Nk-2k-k(k-1)}}{(q^2;q^2)_{N-1}(1-q^{2N})(q^{2N+2};q^2)_k} \\
&=\frac{1}{N}\lim_{q \rightarrow \zeta_N^M}\frac{\sum\limits_{k=0}^{N-1}(1-q^{(2a+1)(2k+1)})\,q^{(2m-1)k^2+2(m-a-1)k}}{1-q^{2N}}
\\
&=\frac{1}{N}\lim_{q \rightarrow \zeta_N^M}\frac{q^{\frac{-(m-a-1)^2}{(2m-1)}}\sum\limits_{k=0}^{(4m-2)N}\chi_{2(2m-1)}^{(a)}(k)\,q^{\frac{k^2}{4(2m-1)}}}{1-q^{2N}}
\\
&=\frac{-1}{8(2m-1)N^2}\zeta_N^{\frac{-M(m-a-1)^2}{(2m-1)}}\sum\limits_{k=0}^{(4m-2)N}k^2\chi_{2(2m-1)}^{(a)}(k)\,\zeta_N^{\frac{Mk^2}{4(2m-1)}}.
\end{align*}
\end{proof}

Theorem \ref{ex3theorem} now follows upon comparing Propositions \ref{ex3prop2} and \ref{ex3prop3}.


\section{Concluding Remarks}
In closing we observe that the evaluations in \eqref{finiteeval1}, \eqref{finiteeval2}, \eqref{finiteeval3}, and \eqref{finiteeval4} can be used to obtain $q$-series identities that hold at roots of unity but not inside the unit disk.   This is in the spirit of quantum $q$-series identities described in \cite{Fo-Me1,Lo1,Lov}.    For instance, if we take $m \mapsto 2m-1$ and $a \mapsto 2a$ in \eqref{finiteeval1} and compare with \eqref{finiteeval4}, we have that if $q$ is a primitive odd $N$th root of unity, then
\begin{equation*}
\begin{aligned}
&\sum_{n_1,\dots, n_{2m-2} =0}^{N-1} (-1)^{n_{2m-2}}q^{\binom{n_{2m-2} + 1}{2} + n_1^2 + \cdots + n_{2m-3}^2 + n_{2a+1} + \cdots + n_{2m-3}} \prod_{i=1}^{2m-3} \begin{bmatrix} n_{i+1} + \delta_{i,2a} \\ n_i \end{bmatrix} \\
&= 2\sum_{n_1, \dots, n_{m-1} =  0}^{N-1} \frac{(-1)^{n_{m-1}}q^{2\binom{n_{m-1}+1}{2}+2n_1^2+\cdots+2n_{m-2}^2+2n_{a+1}+\cdots+ 2n_{m-2}}}{(-q;q^2)_{n_1+\delta_{a,0}}}\prod_{i=1}^{m-2} \begin{bmatrix} n_{i+1}+\delta_{i,a} \\ n_i \end{bmatrix}_{q^2}.
\end{aligned}
\end{equation*}
For another example, comparing the case $m \mapsto 2m-2$ and $a=0$ of \eqref{finiteeval1} with \eqref{finiteeval3}, we have that if $q$ is a primitive odd $N$th root of unity, then
\begin{equation*}
\begin{aligned}
&\sum_{n_1,\dots, n_{2m-3} =0}^{N-1} (-1)^{n_{2m-3}}q^{\binom{n_{2m-3} + 1}{2} + n_1^2 + \cdots + n_{2m-4}^2+n_{1} + \cdots + n_{2m-4}} \prod_{i=1}^{2m-4} \begin{bmatrix} n_{i+1} \\ n_i \end{bmatrix} \\
&= 2\sum_{n_1, \dots, n_{m-1} =  0}^{N-1} \frac{(-1)^{n_{m-1}}q^{2\binom{n_{m-1}+1}{2}+2n_1^2+\cdots+2n_{m-2}^2+2n_1+\cdots+ 2n_{m-2}}(q;q^2)_{n_1}}{(-q)_{2n_1+1}}\prod_{i=1}^{m-2} \begin{bmatrix} n_{i+1} \\ n_i \end{bmatrix}_{q^2}.
\end{aligned}
\end{equation*}
Note that the left-hand sides of the above two equations are essentially Kashaev invariants for certain torus links, as shown by Hikami \cite{Hik}, and so the right-hand sides provide alternative expressions for these invariants.

\section{Acknowledgements}
The authors would like to thank the Unversit\'e Paris Cit\'e for hosting the second author's research visit and both the French Embassy in the United States and the Centre National de la Recherche Scientifique for making this visit possible through a Chateaubriand Fellowship.   They also thank the referee for several important corrections and suggestions.


\end{document}